\documentclass{amsart}

\def\br{\mathbb{R}}
\def\bc{\mathbb{C}}

\DeclareMathOperator{\Real}{\rm Re}
\newcommand{\du}[1]{#1^{\sharp}}
\newcommand{\pd}[1]{#1_{\sharp}}

\newtheorem{theorem}{Theorem}[section]
\newtheorem{lemma}[theorem]{Lemma}
\newtheorem{corollary}[theorem]{Corollary}
\newtheorem{proposition}[theorem]{Proposition}
\theoremstyle{remark}

\theoremstyle{definition}

\numberwithin{equation}{section}

\begin{document}

\title[Weak* continuous states on Banach algebras]{Weak* continuous states on Banach algebras}
 
\date{August, 2008.}

\author{Bojan Magajna}

\address{Department of Mathematics\\ University of Ljubljana\\
Jadranska 21\\ Ljubljana 1000\\ Slovenia}
\email{Bojan.Magajna@fmf.uni-lj.si}

\thanks{The author is grateful to David Blecher for a question which prompt the investigation
presented here and for his remarks.}

\thanks{Partially supported  by the Ministry of Science and Education of Slovenia.}



\begin{abstract}
We  prove that if a unital Banach algebra $A$ is the dual of a Banach space 
$\pd{A}$, then the set of weak* continuous states is weak* dense in the set of all states on $A$. 
Further, weak* continuous states linearly span  $\pd{A}$. 
\end{abstract}
 
\maketitle

\section{Introduction}
An important tool in functional analysis is Goldstine's theorem \cite[3.27]{F}, which says
that for a dual Banach space $A$ the unit ball 
of its predual $\pd{A}$ is weak* dense in the unit ball $B_{\du{A}}$ of the dual $\du{A}$ of 
$A$. 
Given a norm one element $x\in A$, we may consider the set of `states' $S^x(A)=\{\rho\in
B_{\du{A}}:\ \rho(x)=1\}$ and the subset of `normal states' $S^x_n(A)=S^x\cap\pd{A}$. In 
general
$S^x_n(A)$ need not be weak* dense in $S^x(A)$, for $S^x_n(A)$ may even be empty if $x$ does not achieve its norm
as a functional on $\pd{A}$. In this note we show that if $A$ is a unital Banach algebra
and $x=1$ is the unit of $A$ (with $\|1\|=1$), then the set of normal states 
$S_n(A):=S^1_n(A)$ is weak* dense in the set $S(A):=S^1(A)$ of all states. 
Using this, we also show that $S_n(A)$ spans the predual of $A$. Of course, 
all this is well known for von Neumann algebras.
That $S(A)$ spans $\du{A}$ for any unital Banach algebra $A$ was proved by Moore \cite{Mo} 
(\cite{Si1} and \cite{AE} contain simpler proofs). 

Our method is based on a consideration of dissipative elements.  Recall 
that an element $a\in A$ is {\em dissipative} if its numerical range
$W(a):=\{\rho(a):\ \rho\in S(A)\}$ is contained in the left half-plane $\Real z\leq0$. To show
that the set $D_A$ of all such elements is weak* closed, if $A$ is a dual space, we will
need a suitable metric characterization of dissipative elements (Lemma \ref{le21} below). 
A similar, but not the same, characterization was observed
in \cite{BM} for  C$^*$-algebras; however, the argument from \cite{BM} does not
apply to Banach algebras. For each $a\in D_A$ the element $1-a$ is invertible since its
numerical range (hence also its spectrum) is contained in the half-plane $\Real z\geq1$.
We will only need the estimate
\begin{equation}\label{22}\|(1-a)^{-1}\|\leq1\ \ (a\in D_A),
\end{equation}
which is known from the Hille--Yosida theorem on  generators of operator semigroups.
In our present context it can easily be derived
by applying the well known estimate $\|e^{ta}\|\leq1$, $t\geq0$ (see
\cite[p. 55]{BD} or \cite[A13(4)]{T}) to the integral representation 
$(1-a)^{-1}=\int_{0}^{\infty}e^{-t(1-a)}\, dt$
, which  can be verified directly.

\section{Dissipative elements and normal states}

\begin{lemma}\label{le21} An element $a$  of a unital Banach algebra
$A$ is dissipative if and only if 
\begin{equation}\label{21}\|1+ta\|\leq1+t^2\|a\|^2\ \mbox{for all}\ t\geq0.\end{equation}
In particular, if $\|a\|\leq1$, then $a\in D_A$ if and only if $\|1+ta\|\leq 1+t^2$ for all
$t\geq0$. 
\end{lemma}

\begin{proof} If $a$ satisfies (\ref{21}), then
for every state $\rho\in S(A)$ and $t>0$ we have that $|1+t\rho(a)|^2=|\rho(1+ta)|^2
\leq\|1+ta\|^2\leq
(1+t^2\|a\|^2)^2$. This implies that $2\Real\rho(a)\leq t(2\|a\|^2-|\rho(a)|^2)+t^3\|a\|^4$, hence (letting 
$t\to0$) $\Real\rho(a)\leq0$. 

For a proof of the reverse direction,  note that by (\ref{22}) each $a\in D_A$ satisfies
$$\|1+a\|=\|(1-a)^{-1}(1-a^2)\|\leq\|1-a^2\|\leq1+\|a\|^2.$$
But, since $ta$ is also dissipative if $t\geq0$, we may replace a by $ta$ in the last inequality,
which yields (\ref{21}). The last sentence of the lemma follows now easily. 
\end{proof}

In C$^*$-algebras the estimate (\ref{21}) can be improved to $\|1+ta\|^2\leq1+t^2\|a\|^2$ 
($a\in D_A$, $t\geq0$),
a consequence of the C$^*$-identity \cite{BM}. This sharper estimate holds also in some other 
natural examples of Banach algebras, but the author does not know if it holds in general.
Since this topic is not essential for our purposes here, we will
postpone further discussion on it to the end of the paper. 

\begin{theorem}\label{th21}
If a unital Banach algebra $A$ is  a dual Banach space, then $D_A$ is a weak*  
closed subset of $A$. 
Moreover, $S_n(A)$  is  weak* dense in $S(A)$.
\end{theorem}

\begin{proof} The proof is the same as for operator spaces \cite{BM}. Since it is very short, 
we will sketch it here for completeness. Since $D_A$ is convex and 
$tD_A\subseteq D_A$ if $t\geq0$, to prove that $D_A$ is weak* closed,
it suffices to show that the intersection of $D_A$ with the closed unit ball of $A$ is weak*
closed (see e.g. \cite[4.44]{F}). But this follows immediately from Lemma \ref{le21}. 

Denote by $A^{{\sharp}+}$ the set of all 
nonnegative multiples of states on $A$
and by $(D_A)^{\circ}$ the set of all  $\rho\in A^{\sharp}$ such that 
$\Real\rho(a)\leq0$ for all $a\in D_A$. Clearly $A^{\sharp+}\subseteq (D_A)^{\circ}$. 
To prove that $A^{\sharp+}=(D_A)^{\circ}$, let $\rho\in (D_A)^{\circ}$.
Since $it1\in D_A$ for all $t\in\br$ and $-1\in D_A$, it follows that $\rho(1)\geq0$.
Since $a-\|a\|1\in D_A$ for each $a\in A$, we have that $\Real\rho(a)\leq
\|a\|\rho(1)$. Replacing in this inequality $a$ by $zx$ for all $z\in\bc$ with $|z|=1$, 
it follows that 
$|\rho(a)|\leq\|a\|\rho(1)$, hence $\rho\in A^{\sharp+}$. 

Now put $A_{\sharp}^+=A_{\sharp}\cap A^{{\sharp}+}$ and $(D_A)_{\circ}=
(D_A)^{\circ}\cap A_{\sharp}$.
Then $A_{\sharp}^+=(D_A)_{\circ}$.
Since $D_A$ is weak* closed, a bipolar type argument  shows that
$D_A=((D_A)_{\circ})^{\circ}$  and that 
$(D_A)_{\circ}$ 
is weak* dense in $(D_A)^{\circ}$. This means that $A_{\sharp}^+$
is weak* dense in $A^{\sharp+}$. Now it follows easily  that 
$S_n(A)$ is weak* dense in $S(A)$.
\end{proof}

\begin{corollary}\label{co1} Let $A$ be as in Theorem \ref{th21}. For every closed convex 
subset $C$ of $\bc$ the set $A_C=\{a\in A:\ W(a)\subseteq C\}$ is weak* closed in $A$.
\end{corollary}

\begin{proof} Since $C$ is the intersection of half-planes containing it, this follows
from the fact that $A_{\{\Real z\leq0\}}=D_A$ is weak* closed.
\end{proof}

\begin{theorem}\label{th22} If $A$ is as in Theorem \ref{th21}, then $S_n(A)$ spans 
$\pd{A}$. Each $\omega\in\pd{A}$ with $\|\omega\|<(e\sqrt{2})^{-1}$ (where $\log e=1$) can be 
written as $\omega=
t_1\omega_1-t_2\omega_2+i(t_3\omega_3-t_4\omega_4)$, where $\omega_j\in S_n(A)$ and $t_j\in
[0,1]$.
\end{theorem}

\begin{proof} Put $S=S(A)$ and $S_n=S_n(A)$. For a subset $V$ of $A$ define the polar $V^{\diamond}$ by
$V^{\diamond}=\{\rho\in\du{A}:\ |\Real\rho(a)|\leq1\ \forall a\in V\}$. In the same way define
also polars of subsets of $\pd{A}$ and `prepolars' $V_{\diamond}$ of subsets of $A$ or
$\du{A}$. Let $U=S_{\diamond}$. Then $U$ is the set of all elements $a\in A$ with the numerical
range contained in the strip $|\Real z|\leq1$, hence  $U$ is
weak* closed by Corollary  \ref{co1}.  Since $S_n$ is weak* dense in $S_n$ by Theorem 
\ref{th21} and $U=S_{\diamond}$, it follows that $U=S_n^{\diamond}$, hence by the bipolar
theorem $U_{\diamond}=\overline{\rm co}(-S_n\cup S_n)$. Let $V=iU\cap U$.
Then $V_{\diamond}$ is equal to the norm closure of the 
convex hull of $(iU)_{\diamond}\cup
U_{\diamond}$, hence (since $(iU)_{\diamond}=-iU_{\diamond}=-i\overline{\rm co}(-S_n\cup S_n)$)
$V_{\diamond}$ is the closure of the set $S_0:={\rm co}(S_n\cup(-S_n)\cup(iS_n)\cup(-iS_n))$.
On the other hand, by the definition of $U$, $V$ is just the set of all $a\in A$ with
the numerical range $W(a)$ contained in the square $[-1,1]\times[-i,i]$. Since  for every $a\in A$
the inequality $\|a\|\leq ew(a)$ holds, where $w(a)$ is the numerical radius of $a$
(see \cite{BD} or \cite[2.6.4]{P}), it follows that $V$ is contained in the closed ball $(\sqrt{2}e)B_{A}$ of $A$ 
with the center $0$
and radius $\sqrt{2}e$. Consequently $\overline{S_0}=V_{\diamond}\supseteq
(\sqrt{2}e)^{-1}B_{\pd{A}}$.

Let $T=\{t\omega:\ \omega\in S_n,\ t\in[0,1]\}$. Since $S_n$ is norm closed and bounded and the interval 
$[0,1]$ is compact,
it is not hard to verify that $T$ is norm closed. Further, since $S_n\subseteq T$ is convex and
$tT\subseteq T$ for all
$t\in[0,1]$, it follows from the definition of $S_0$ that
$S_0\subseteq T_0:=T-T+iT-iT$. Therefore we conclude from the previous paragraph that $B_{\pd{A}}
\subseteq \sqrt{2}e\overline{T_0}$. Thus, given  $\omega\in \pd{A}$ and $\delta\in(0,1)$,
there exists $\omega_0\in \|\omega\|\sqrt{2}eT_0$ such that
$\|\omega-\omega_0\|<\delta.$
Applying the same  to  $\omega-\omega_0$, we find
$\omega_1\in \delta\sqrt{2}e T_0$ such that $\|\omega-\omega_0-\omega_1\|\leq\delta^2$. 
Continuing, we find a sequence of functionals $\omega_n\in\delta^n\sqrt{2}eT_0$ such that
$$\|\omega-\omega_0-\ldots-\omega_n\|\leq\delta^{n+1}.$$
Thus $\omega=\omega_0+\sum_{n=1}^{\infty}\omega_n$ is of the form
\begin{equation}\label{1}\omega=\sqrt{2}e(\|\omega\|\rho_0+\sum_{n=1}^{\infty}\delta^{n}\rho_n),
\end{equation}
where $\rho_n\in T_0$. By the definition of $T_0$ we have that $\rho_n=\rho_{n,1}-\rho_{n,2}
+i(\rho_{n,3}-\rho_{n,4})$, where $\rho_{n,j}\in T={\rm co}(\{0\}\cup S_n)$. Put 
$\gamma=\|\omega\|+\sum_{n=1}^{\infty}\delta^n=\|\omega\|+\delta(1-\delta)^{-1}$. Since $T$ is 
closed and convex, 
$\psi_j:=\gamma^{-1}(\|\omega\|\rho_{0,j}+\sum_{n=1}^{\infty}\delta^{n}\rho_{n,j})\in T$ 
for each $j$.
From (\ref{1}) we have now
\begin{equation}\label{23}\omega=(\sqrt{2}e)\gamma(\psi_1-\psi_2+i(\psi_3-\psi_4)),\end{equation}
a linear combination of normal states. If $\|\omega\|<(e\sqrt{2})^{-1}$, we may choose
$\delta$ so small that $\gamma\leq(e\sqrt{2})^{-1}$, and then we conclude from (\ref{23}) that $\omega$ is of the form
$\omega=\sum_{j=0}^3t_ji^j\omega_j$, where $\omega_j\in S_n$ and $t_j\in[0,1]$.
\end{proof}

The well-known characterizations of hermitian elements in a Banach algebra
\cite{V}, \cite{P}, \cite{BD} do not seem to imply easily that  
the real subspace $A^h$ of all such elements is weak* closed, if $A$ is a dual Banach space.
For this reason we state here another simple characterization. In the case of C$^*$-algebras
this characterization has been observed earlier by others and demonstrated to be 
useful \cite{BN}.

\begin{proposition} An element $h$ in a unital Banach algebra $A$ is hermitian if and only if
\begin{equation}\label{3}\|h+it1\|^2\leq\|h\|^2+t^2\ \mbox{for all}\ t\in\br.
\end{equation}
Thus, if $A$ is a dual Banach space, then $A^h$ is a weak* closed subset of $A$.
\end{proposition}

\begin{proof}If (\ref{3}) holds, then for each $\rho\in S(A)$ 
\begin{equation}\label{4}|\rho(h)+it|^2=|\rho(h+it1)|^2\leq\|h\|^2+t^2\ \ (t\in\br),
\end{equation} 
which implies (by letting $t\to\infty$) that $\rho(h)\in\br$, hence
$h$ is hermitian. Conversely, if $h$ is hermitian then (\ref{4}) holds. But, by a result of
Sinclair \cite{Si} the norm of an element of the form $a=h+\lambda1$, where $h$ is hermitian
and $\lambda\in\bc$, is equal to the spectral radius $r(a)$, hence also to the numerical 
radius $w(a)$ (since in general $r(a)\leq w(a)\leq\|a\|$). Thus,
taking in (\ref{4}) the supremum over all states $\rho\in S(A)$, we get (\ref{3}).

The above argument also shows that a contraction is hermitian if and only if $\|h+it1\|^2\leq
1+t^2$ for all $t\geq0$. This imlies that the intersection $A^h\cap B_A$ is weak* closed, hence
$A^h$ is weak* closed by the Krein--Smulian theorem.
\end{proof}
 
The estimate 
\begin{equation}\label{30}\|1+a\|^2\leq1+\|a\|^2,\end{equation}
which holds for all dissipative elements in C$^*$-algebras, holds in general Banach algebras
at least for dissipative elements of a special form. For example, using the fact that for
hermitian elements the norm is equal to the spectral radius, it can be shown that (\ref{30}) holds
for dissipative hermitian elements.

\begin{proposition} In any unital Banach algebra $A$ each dissipative element of the form 
$a=-tp$, where $p$ is an idempotent and $t\in(0,\infty)$, satisfies the estimate (\ref{30}).
\end{proposition}

\begin{proof}  By the well known criterion \cite[p. 55]{BD} $-p$ is dissipative
if and only if $\|e^{-tp}\|\leq1$ for all $t\geq0$. From the Taylor series expansion of
$e^{-tp}$ we compute that $e^{-tp}=(1-s)q+s1$, where $q=1-p$ and $s=e^{-t}$. So, $-p$ is dissipative
if and only if $\|(1-s)q+s1\|\leq1$ for all $s\in[0,1]$, which is equivalent to $\|q\|\leq1$. 
Since the norm of any nonzero idempotent is at least $1$, we conclude that 
$-p$ (hence also $-sp$, if $s>0$) is dissipative if and only if $\|q\|=1$  

To prove (\ref{30}), where $a=-tp$ is assumed dissipative (thus $\|q\|=1$), consider first the case when $t\in(0,1]$. Put $s=1-t$
and note that 
$$\|1-tp\|^2=\|s1+(1-s)q\|^2\leq (s+(1-s)\|q\|)^2=1\leq1+t^2\|p\|^2.$$
On the other hand, if $t>1$, then
$$\begin{array}{lll}\|1-tp\|^2&=&\|q+(1-t)p\|^2\leq(1+(t-1)\|p\|)^2\\
 &=&1+t^2\|p\|^2-2t\|p\|(\|p\|-1)-\|p\|(2-\|p\|)\leq1+t^2\|p\|^2,
 \end{array}$$
since $1\leq\|p\|=\|1-q\|\le2$. 
\end{proof}

{\em Question.} Do all dissipative elements in each unital Banach algebra satisfy (\ref{30})?

\end{document}